\documentclass{amsart}

\usepackage{amsmath}
\usepackage{amsfonts}
  \usepackage{paralist}
  \usepackage{graphics} 
  \usepackage{epsfig} 

 \usepackage[colorlinks=true]{hyperref}

\newtheorem{theorem}{Theorem}[section]
\newtheorem*{theorem*}{Theorem}

\newtheorem{lemma}[theorem]{Lemma}
\newtheorem{proposition}{Proposition}

\theoremstyle{definition}

\newtheorem{remark}{Remark}

\newtheorem{example}{Example}

\newcommand{{\cstm}}{{\textit{$\phi$-cstm}}}
\newcommand{{\ecstm}}{{\textit{e$\phi$-cstm}}}
\def\CA{\mathcal{A}}
\def\CB{\mathcal{B}}
\def\CC{\mathcal{C}}
\def\CD{\mathcal{D}}
\def\CP{\mathcal{P}}
\def\CL{\mathcal{L}}
\def\CT{\mathcal{T}}
\def\wt{\widetilde}
\def\wh{\widehat}
\def\8{\infty}
\def\ninf{{n\to\infty}}
\def\pardef{:=}
\def\ul{\underline}
\def\ol{\overline}
\def\S{\Sigma}
\def\s{\sigma}
\def\BBone{\mathbf{1}}
\def\R{\mathbb{R}}
\def\Z{\mathbb{Z}}
\def\disp{\displaystyle}
\def\eps{\varepsilon}
\def\llb{[}
\def\rrb{]}
\newcommand{\pA}{{\partial A}}
\newcommand{\pinf}{{+\8}}

\def\Psiin{\Psi_{\mathrm{in}}}
\def\Psiout{\Psi_{\mathrm{out}}}
\def\alphain{\alpha_{\mathrm{in}}}
\def\Kac{Ka$\mathrm{\check{c}}\ $}        

\title[Large deviations for return times]
      {Large deviations for return times in non-rectangle sets for axiom A diffeomorphisms}

\author[Renaud Leplaideur and Beno\^{\i}t Saussol]{}

\date{December 2007}


 \email{renaud.leplaideur@univ-brest.fr}
 \email{benoit.saussol@univ-brest.fr}

\begin{document}
\maketitle

\centerline{\scshape Renaud Leplaideur }

\medskip

\centerline{\scshape Beno\^\i t Saussol }
\medskip
{\footnotesize
 \centerline{Laboratoire de Math\'ematiques, CNRS UMR 6205,  Universit\'e de Bretagne Occidentale, 6 av. Victor Le Gorgeu }
   \centerline{CS 93837, 29238 BREST Cedex 3, FRANCE}
} %

\bigskip

 \centerline{(Communicated by the associate editor name)}

\begin{abstract}
For Axiom A diffeomorphisms and equilibrium states, we prove a Large deviations result for the sequence of successive return times into a fixed Borel set, under some assumption on the boundary.
Our result relies on and extends the work by Chazottes and Leplaideur who considered cylinder sets of a Markov partition. 
\end{abstract}

\section{Introduction}

Recall that for any given measurable and ergodic dynamical system $(X,T,\mu)$, and for any set $A$ with positive $\mu$-measure, \Kac's lemma together with birkhoff ergodic theorem implies that the sequence $r_A^n$ of $n$th return-times into $A$ by iterations of the map $T$ satisfies 
$$\lim_\ninf \frac{r^n_A(x)}{n}=\frac{1}{\mu(A)}
\quad\text{for $\mu$-a.e. $x$}.
$$

We are interested in fluctuations of order $n$ of $r_A^n$ around $\frac{n}{\mu(A)}$, so we want to prove a Large Deviations Principle (LDP for short), that is, we want to show the existence of the \emph{rate function $\Phi_A$} such that for every $u\ge \frac{1}{\mu(A)}$,
$$\lim_{n\to \infty}\frac1n\log\mu\left\{\frac{r^n_A}n\geq u\right\}=\Phi_A(u)
$$
and for every $\disp 0\leq u\le\frac{1}{\mu(A)}$,
$$\lim_{n\to \infty}\frac1n\log\mu\left\{\frac{r_A^n}n\leq
u\right\}= \Phi_A(u).
$$
If this holds, we will say that the sequence of return-times into the set $A$ satisfies the LDP for the measure $\mu$. 
If the above LDP only holds for $u\subset ]\ul u,\ol u[$  with $\ul u<\frac{1}{\mu(A)}<\ol u$, we will say that the sequence of return-times into $A$ satisfies a LDP for the measure $\mu$ \emph{near the average}.

A standard method to get a LDP is to prove the existence of the scaled-\emph{cumulant generating function $\Psi_A$}, defined by the following limit
\[
\Psi_A(\alpha)=\lim_{n\to\8} \frac1n \log \int e^{\alpha r_A^n}d\mu,
\]
and prove that it is differentiable. In this case it is well known that the rate function exists and these two
 functions form a Legendre transform pair, namely
\begin{equation}\label{equ-psidonnephi}
\Phi_A(u)=\inf_{\alpha}\left\{ \Psi_A(\alpha)-\alpha u\right\}.
\end{equation}
This is the approach that we are adopting in this paper, with the difference that we do not prove the differentiability of the function $\Psi_A$. It is noteworthy that with our monotone approximation method
this assumption is not required (see Proposition~\ref{pro-inegalout} for details). 

Our result applies to Axiom A diffeomorphism and equilibrium state of H\"older potential, 
and for the sequence of successive return times into a Borel set $A$ that satisfies a condition 
about the smallness of its (non-Markovian) boundary. Taking a Markov partition and the corresponding semi-conjugacy, we will state the result for subshifts of finite type, keeping in mind that this really corresponds to a result for Borel sets on the manifold.


\section{Statements}

Throughout, $(\Sigma,\sigma)$ will denote a topologically mixing subshift of
 finite type.  The set of vertices of the defining graph of $(\Sigma,\sigma)$ is
 $\{1,\ldots,N\}$ with $N\geq 2$. We denote by 
$\CA=(a_{ij})$ the $N\!\times\! N$-transition (irreducible and aperiodic) matrix associated to $\Sigma$;
namely  points in $\Sigma$ are sequences ${x}=\{x_n\}_{n\in\Z}$ such that
for every $n$, $x_n$ belongs to $\{1,\ldots,N\}$ and
$$a_{x_n x_{n+1}}=1.$$ 
Recall that if $f$ is an Axiom A diffeomorphism of a compact manifold $M$ then there always exists a subshift of finite type $\Sigma$ and a coding map $\pi\colon \Sigma\to M$ such that $\pi\circ\sigma=f\circ\pi$.
 
Let $\phi:\S\rightarrow\R$  be $\alpha$-H\"older continuous. For a given $\s$-invariant measure $\lambda$, the $\phi$-pressure  is the quantity  $P_\lambda(\phi):=h_\lambda(\s)+\int\phi\,d\lambda$; $P_\lambda(\phi)$ will also be called the $\lambda$-pressure of $\phi$. The unique \emph{equilibrium state} for $\phi$, {\it i.e.} the unique $\s$-invariant probability measure with maximal $\phi$-pressure, will be denoted by $\mu_\phi$. Its pressure is the topological $\phi$-pressure.
 
For a set $A\subset \S$ and an integer $n$, we denote by $\partial A$ its topological boundary $\ol A\cap \ol{\S\setminus A}$. 

Note that $\pA$ can be empty; this holds for example when $A$ is a finite union of cylinders. We let $\wt\CP_\phi(\pA)$ be the $\phi$ pressure of $\pA$ ; since $\pA$ may not be invariant we define it according to the variational principle:
\[
\wt\CP_\phi(\pA)=\sup \left\{ h_\nu(\s)+\int \phi d\nu\colon \nu\text{ ergodic and }\nu(\pA)>0\right\}
\]
Note that this does not correspond to the dimension-like definition of the pressure introduced by Pesin and Pitskel~\cite{pesin-pitskel}.

If $D$ is any subset in $\S$, and for $x\in \S$, we denote by $r_D(x)$ the first hitting in $D$ by iterations of $\s$ (if it exists). Namely $r_D(x)$ is the smallest integer $n>1$ such that $\s^n(x)$ belongs to $D$, and $r_D(x)=+\8$ if no such integer exists. We also set $r^1_D(x)=r_D(x)$, and denote the $n$th return time $r^n_D(x)$ the cocycle defined by 
$$r^{n+1}_D(x)=r^n_D(x)+r_D(\s^{r^n_D(x)}(x)).$$

Let $\max(A):=\sup_\mu \mu(A)$ and $\min(A)=\inf_\mu\mu(A)$, where the extrema are taken among all invariant measures $\mu$. Then our result is: 

\begin{theorem*}
Let $\phi:\S\rightarrow\R$ be any H\"older continuous function.
Let $A\subset \S$ be a Borel set.  We have:
\begin{enumerate}
\item\label{point1} if for any $\s$-invariant measure $\mu$, $\mu(\pA)=0$, then the sequence $(r^n_A)_{n\geq 1}$ satisfies the Large Deviations Principle for $\mu_\phi$, except possibly for discontinuity points. More precisely, for any $u\in\left(\frac{1}{\max(A)},\frac{1}{\min(A)}\right)$ the rate function $\Phi_A(u)$ exists and is finite, and for any $u$ outside the interval $[\frac{1}{\max(A)},\frac{1}{\min(A)}]$, the rate function $\Phi_A(u)$ exists is equal to $-\infty$.

\item\label{point2} if the $\phi$-pressure $\wt\CP_\phi(\pA)$ of the boundary is strictly smaller than the $\phi$-topological pressure then the sequence $(r^n_A)_{n\geq 1}$ satisfies a  Large Deviations Principle for $\mu_\phi$ \emph{near the average}.
\end{enumerate}
\end{theorem*}

We recall that our method is based on the existence of the cumulant generating function
$\Psi_A(\alpha)$ for every $\alpha$ in some open interval $(\ul\alpha,\ol\alpha)$. For the statement~\ref{point1}, we will prove that the function $\Psi_A$ is defined on an interval $(-\8,\ol\alpha)$.
For the statement~\ref{point2}, we will only get the existence of $\Psi_A$ on some open neighborhood $(\ul\alpha,\ol\alpha)$ of $0$.

We emphasize that if $f$ is an axiom A diffeomorphism of a manifold $M$ and $\mu_\phi$ is an equilibrium state of a H\"older potential $\phi$, and $V$ is a Borel set then the theorem applies to $A=\pi^{-1}V$ since $\partial\pi^{-1}V\subset\pi^{-1}\partial V$ under the same hypotheses on $\partial V$.
The hypothesis in statement~\ref{point1} could seem very restrictive;
however it should be satisfied quite often in some general situations, as the following example suggests.
\begin{example}
Let $(M,f)$ be an hyperbolic automorphism of the 2-torus, and consider the family of balls $B(x,r)$ about a given point $x\in M$. Then, for all but countably many radii $r>0$, the condition $\mu(\partial B(x,r))=0$ for every invariant measure $\mu$ is satisfied.
\end{example}
\begin{proof}
Let $S$ be the boundary of a ball. Using hyperbolicity one can show that the intersection of $S$ with its images $f^n(S)$ consists, at most, of countably many points. 
Hence the set of recurrent points $R(S)$ in $S$ is at most countable. 
If an invariant measure gives weight to $S$, by Poincar\'e Recurrence Theorem it implies that it gives weight to $R(S)$ which is a countable set. Thus the measure must have an atomic part consisting of a periodic orbit. Hence $S$ must contain a periodic point.
Since the set of periodic points of such a map is countable there can be at most countably many boundaries $\partial B(x,r)$ carrying a periodic point as $r$ varies, which proves the proposition.
\end{proof}

The hypothesis in statement~\ref{point2} about the pressure of the boundary appears quite naturally in the thermodynamic formalism of dynamical systems with singularities. In the case of $\phi=0$ it simply says that the boundary $\pA$ does not carry full measure theoretical entropy. 
A more explicit condition can be given on the manifold itself.
\begin{proposition}\label{pro-bdpress-mfd}
Let $f$ be an axiom A diffeomorphism of a manifold $M$ and let $\mu_\phi$ be an equilibrium state of a H\"older potential $\phi$. Let $V$ be a Borel set and denote by $U_\eps(\partial V)$
the $\eps$-neighborhood of the boundary $\partial V$. 
Assume that there exist some constants $c>0$ and $a>0$ such that 
\[
\mu_\phi(U_\eps(\partial V)) \le c \eps^a \quad \forall \eps>0.
\]
Then, $\wt\CP_\phi(\partial V)<\CP_\phi(M)$. In particular the sequence of return times into $V$ satisfies a Large Deviations Principle near the average.
\end{proposition}
See Section~\ref{sec-proof-th2} for further details.
In particular we have:
\begin{example}
Let $(M,f)$ be a $C^2$ \emph{volume preserving} Anosov diffeomorphism and let $V\subset M$ be an Borel set with piecewise $C^1$ boundary. Then the sequence of return times into $V$ satisfies a Large Deviations Principle near the average.
\end{example}
\begin{proof}
Set $\phi=-\log |D_x^uf|$. The equilibrium measure $\mu_\phi$ is the SBR measure which is here the volume measure, and the assumption in Proposition~\ref{pro-bdpress-mfd} is clearly satisfied.
\end{proof}

\medskip
\textit{Outline of the proof of the theorem: } in Section \ref{sec-CB-CC} we recall how the LDP was obtained for the return-times in cylinders. 
In Section~\ref{sec-exis-psiA} we compare the cumulant generating functions of inner and outer approximation of our set $A$ by union of cylinders. 
In section~\ref{sec-proof-th2}, under the assumption of the statement~\ref{point2} of the theorem,  we prove the existence of the cumulant generating function $\Psi_A$ on some interval. 
In section~\ref{sec-proof-th1} we give a dynamical proof of the first statement of the theorem.

\section{Large deviations for return time in cylinders}\label{sec-CB-CC}

We first recall the local thermodynamic formalism introduced in \cite{leplaideur1}. Then we recall how the large deviations principle for union of cylinders was obtained in \cite{chazottes-leplaideur}. Finally, we derive a uniform mass concentration principle. 


\subsection{Induced systems and local thermodynamic formalism}\label{subsec-thermo-loc}

For a given point $x=(x_n)_{n\in\Z}\in \S$, the past (resp. future) of the point denotes the backward (resp. forward) sequence  $(x_n)_{n\leq 0}$ (resp. $(x_n)_{n\geq 0}$).  
For $x$ and $y$ in $\S$, when $x_0=y_0$, the point $z\disp\pardef\llb x,y\rrb$ is the point obtained when we take the past of $y$ and the future of $x$. 

In $\Sigma$, the {\it cylinder} $[i_k,\ldots,i_{k+n}]$
will denote the set of points $x\in\Sigma$ such that $x_j=i_j$ (for every $k\leq
j\leq k+n$). Such a cylinder will also be called a word (of length
$n+1$) or equivalently a $(k,k+n)$-cylinder. 
If $x$ is in $\Sigma$, $C_{k,k+n}(x)$ will denote the
cylinder $[i_k,\ldots,i_{k+n}]$ such that $x_j=i_j$ (for every $k\leq
j\leq k+n$). By extension, $C_{-\8,n}(x)$ will denotes the set of points $(y_k)$ such that $y_k=x_k$ for every $k\leq n$; similarly $C_{n,+\8}(x)$ will denotes the set of points $(y_k)$ such that $y_k=x_k$ for every $k\geq n$. By definition, the local unstable leaf $W^u_{loc}(x)$ is $C_{-\8,0}(x)$, and the local stable leaf $W^s_{loc}(x)$ is $C_{0,+\8}(x)$. For $n\geq 0$, a $n$-cylinder will denote a $(-n,n)$-cylinder. 
The letter $R=\cup R_i$ denotes some finite union of $(-n,n)$-cylinders; in each of these cylinders we fix some local unstable leaf $F_i$. There is a natural projection from each $R_i$ onto each $F_i$ defined by  $\pi_{F_i}(z)=\llb z,x\rrb$, where $x$ is any point in $F_i$. For convenience we denote by $\pi_F$ the map defined on $R$ by 
$$\pi_F(z)=\pi_{F_i}(z)\iff z\in R_i.$$
We  denote by $g$ the first return map in $R$, and by $g_F$ the map $\pi_F\circ g$.  We thus have $g(x)=\s^{r_R(x)}(x)$. Note that if the maps $r_R$, $g$ and $g_F$ are not defined everywhere in $R$, the inverse branches of $g_F$ are well defined in the whole $F$. 

We can thus define the Ruelle-Perron-Frobenius operator for $g_F$: for $x$ in $F$,  we set 
$$\CL_S(\CT)(x)=\sum_{y,\  g_F(y)=x}e^{S_{r_R(y)}(\phi)(y)-r_R(y) S}\CT(y),$$
where $\CT:F\rightarrow\R$ is a continuous function, and $S$ is a real parameter.  As usual, $S_n(\phi)(x)$ denotes the Birkhoff sum $\phi(x)+\cdots+\phi\circ\s^{n-1}(x)$.

There exists some critical $S_c$, such that for every $S>S_c$ all the following holds: $\CL_S$ admits a unique and single dominating eigenvalue $\lambda_S$ in the set of $\alpha$-H\"older continuous functions. 
The adjoint operator $\CL_S^*$ has also $\lambda_S$ for unique and single dominating eigenvalue;  we denote by $\nu_S$ the unique probability measure on $F$ such that $\CL_S^*(\nu_S)=\lambda_S \nu_S$. We denote by $H_S$, the unique $\alpha$-H\"older continuous and positive function on $F$ satisfying $\CL_S(H_S)=\lambda_S H_S$ and $\int H_S\,d\nu_S=1$.
We also denote by $\mu_S$ the measure $H_S \nu_S$, and by $\wh\mu_S$ the natural extension of $\mu_S$. We recall that $\mu_S$ is a $g_F$-invariant probability measure, and $\wh\mu_S$ is a $g$-invariant probability measure. At last, we denote by $m_S$ the opened-out measure: namely $m_S$ is the $\s$-invariant measure satisfying, $m_S(R)>0$, and $\wh\mu_S$ is the conditional measure $m_S(\cdot|R)$. 

The spectral properties of $\CL_S$ yield the existence of  positive real constants $C_\phi$ and $\eps_S$,  such that for every  H\"older continuous $\CT:F\rightarrow\R$, for every integer $n\geq 1$ and for every $x$ in $F$
\begin{equation}\label{eq1-trouspect-perron}
\CL_S^n(\CT)(x)=e^{n\log\lambda_S}\int\CT\,d\nu_SH_S(x)+O(e^{n(\log\lambda_S-\eps_S)})||\CT||_\8.
\end{equation}
Note that $H_S$ is a positive function on the compact set $F$. 


In \cite{chazottes-leplaideur}, it is proved that the critical value $S_c$ is the pressure of the dotted system, with hole $R$, associated to the potential $\phi$. Namely we consider in $\S$ the system $\S_R:=\bigcap_{n\in\Z}\s^{-n}(\S\setminus R)$. Although not explicitly mentioned, the case $\Sigma_R=\emptyset$ appears, when $\min(R)\neq0$. In this case one simply has $S_c=-\infty$ and the identity remains valid with the convention that the pressure of the emptyset is $-\infty$.
The proof in~\cite{chazottes-leplaideur} was done under the assumption that $\Sigma_R$ is mixing. We claim that the mixing hypothesis  can be omitted.  Indeed, any subshift of finite type can be decomposed in irreducible components, which satisfy the mixing property, but for some iteration of the map $\s$ (see e.g.~\cite{bowen}). As we are considering first returns in $R$, note that the word defined by the cylinder $C_{0,r_R(x)}(x)$ contains no $R_i$ but at the first position. 
Now, two different irreducible components can be joined in $\S$ only by a path which contains $R$. Therefore, the word  defined by the cylinder $C_{0,r_R(x)}(x)$ is an admissible word for a unique irreducible component.

Unicity of the equilibrium state in any mixing subshift (for $\phi$) implies that the topological $\phi$-pressure $\CP_\phi(\S_R)$ for $(\S_R,\s)$ is strictly lower than the topological $\phi$-pressure for $\S$, $\CP_\phi(\S)$.

We now finish this subsection with some important characterization for the measure $m_S$. 
\begin{lemma}\label{lem-carac-mus}
The measure $m_S$ is the unique equilibrium state in $(\S,\s)$ associated to $\phi-\log\lambda_S \BBone_R$. Moreover, its $\phi-\log\lambda_S \BBone_R$-pressure is $S$. 
\end{lemma}
\begin{proof}
For simplicity we set $\beta:=\log\lambda_S$.
The measure $m_S$ satisfies,
\begin{equation}\label{equ2-preuvecvms}
h_{m_S}(\s)+\int\phi\,dm_S=S+m_S(R)\beta.
\end{equation}
We refer the reader to \cite{leplaideur1}, Proposition~6.8 for a proof.  Moreover, the measure $\wh\mu_S$ is the unique equilibrium state for $(R,g)$ associated to the potential $S_{r(\cdot)}(\phi)(\cdot)-S r(\cdot)$, with pressure $\beta$. Let us pick some $\s$-invariant probability measure $\nu$.

Let us first assume that $\nu(R)>0$. We have
\begin{eqnarray*}
h_\nu(\s)+\int\phi\,d\nu-S&=&\nu(R)\left(h_{\nu_{|R}}(g)+\int S_{r(\cdot)}(\phi)\,d\nu_{|R}-S \int r(\cdot)\,d\nu_{|R}\right),\\
&\leq&\nu(R)\beta,
\end{eqnarray*}
where $\nu_{|R}$ is the conditional measure $\nu(\cdot|R)$. This gives 
$$h_\nu(\s)+\int\phi\,d\nu-\beta \int\BBone_R\,d\nu\leq S,$$
with equality if and only if $\nu_{|R}=\wh\mu_S$ ({\it i.e.} $m_S=\nu$).

If we assume that $\nu(R)=0$, then $\nu$ is a $\s$-invariant probability measure with support in $\S_R$. Therefore it must satisfy 
$$h_\nu(\s)+\int\phi\,d\nu-\beta \int\BBone_R\,d\nu=h_\nu(\s)+\int\phi\,d\nu\leq S_c<S.$$
This finishes the proof of the lemma.
\end{proof}

\subsection{Large deviations for return times in cylinders}

In \cite{chazottes-leplaideur}, it is also proved that $\lambda_S\rightarrow+\8$ as $S$ goes to $S_c$. Moreover, the map $S\mapsto \log\lambda_S$ is a decreasing convex map on $]S_c,+\8[$. There also exists some complex neighborhood of $]S_c,+\8[$ such that the map $S\mapsto \log\lambda_S$ admits an analytic continuation on it. In particular the map $S\mapsto\log\lambda_S$ is real-analytic on $]S_c,+\8[$.

Finally, it is proved in \cite{chazottes-leplaideur} that for every $\alpha<\alpha(R):=\CP_\phi(\S)-\CP_\phi(\S_R)$,
\begin{equation}\label{eq2-gd-rectangle}
\lim_\ninf\frac1n\log\int_R e^{\alpha r^n_R(x)}\,d\mu_\phi=\log\lambda_{\CP_\phi(\S)-\alpha}.
\end{equation}

We shall show now that the large deviations for successive return time and entrance time is the same question; namely, the fact that we are starting from the set $R$ or from the whole space to compute the integral does not make any difference.
\begin{proposition}
If $R$ and $S$ are non-empty finite unions of cylinders, then 
\[
\Psi_R(\alpha)=\lim\frac1n\log \int_{S} e^{\alpha r_R^n}d\mu_\phi,
\]
in particular we have $\Psi_R(\alpha)=\log\lambda_{\CP_\phi(\S)-\alpha}$.
\end{proposition}

\begin{remark}\label{rem-ldprectangle}
As mentioned in the introduction, this readily implies the large deviations principle for return times in the form given by~\eqref{equ-psidonnephi} since the function $\Psi_R$ is differentiable.
\end{remark}

The proposition is a weak consequence of the $\psi$-mixing property of the measure $\mu_\phi$. Indeed, there exists $M>0$ and $\kappa>1$ such that if $f$ and $g$ are two integrable functions such that $f(x)$ only depends on $(x_n)_{n\le p}$ and $g(x)$ only depends on $(x_n)_{n\ge p+M}$ then
\begin{equation}\label{equ-psimixquivabien}
\kappa^{-1}\int f d\mu_\phi\int gd\mu_\phi\le\int fg d\mu_\phi \le \kappa \int fd\mu_\phi \int g d\mu_\phi.
\end{equation}
\begin{lemma}
If $R$ and $S$ are finite union of $(-m,m)$ cylinder then for any $n\ge M+2m$ and for 
\[
\begin{split}
\kappa^{-1}\mu(S)\int_\Sigma e^{\alpha r_R^{n-(M+2m)}} d\mu_\phi
&\le \int_S e^{\alpha r_R^n} d\mu_\phi\le 
\kappa\mu(S)e^{\alpha(M+2m)}\int_\Sigma e^{\alpha r_R^{n}} d\mu_\phi
\quad(\alpha\ge0)\\
\kappa^{-1}\mu(S)e^{\alpha(M+2m)}\int_\Sigma e^{\alpha r_R^{n}} d\mu_\phi
&\le \int_S e^{\alpha r_R^n} d\mu_\phi\le 
\kappa\mu(S)\int_\Sigma  e^{\alpha r_R^{n-(M+2m)}}d\mu_\phi
\quad(\alpha\le0).
\end{split}
\]
\end{lemma}
\begin{proof}
For any $n\ge M+2m$ we have
\[
r_R^{n-(M+2m)}\circ f^{M+2m} \le r_R^n \le M+2m+r_R^{n}\circ f^{M+2m}
\]
from which the result follows by inequality \eqref{equ-psimixquivabien}. 
\end{proof}
The proof of the proposition consists in applying twice the lemma : from the integration over $S$ to $\Sigma$ and then to $R$.

\subsection{Concentration of the mass}

Let $R$ be a finite union of cylinders. The large deviations principle holds for the return times $r^n_R$. It is well-known that this implies a kind of concentration of the mass. 

\begin{proposition}\label{prop-concen-mass}
Let $R$ be a finite union of cylinders. Let $\alpha$ and $\delta>0$ such that $\alpha+\delta<\alpha(R)$.
Then for every $\tau>\displaystyle\frac{\Psi_R(\alpha+\delta)-\Psi_R(\delta)}{\delta}$ we have
$$\Psi_{R}(\alpha)=\lim_\ninf\frac1n\log\int e^{\alpha r^n_{R}} d\mu_\phi=\lim_\ninf\frac1n\log\int_{\left\{r^n_{R}\leq n\tau\right\}}e^{\alpha r^n_{R}}d\mu_\phi.$$
\end{proposition}
\begin{proof}
Take $\eps>0$ so small that $-\delta\tau+\Psi_R(\alpha+\delta)+\eps\le\Psi_R(\alpha)-\eps$.
Using Markov inequality we get 
\begin{eqnarray*}
   \int_{r^n_{R}>n\tau}e^{\alpha r^n_{R}}d\mu_\phi&=&
\int_{r^n_{R}>n\tau}e^{(\alpha+\delta) r^n_{R}}e^{-\delta r^n_{R}}d\mu_\phi\\
&\leq &e^{-\delta n\tau}\int_{r^n_{R}>n\tau}e^{(\alpha+\delta) r^n_{R}}d\mu_\phi\\
&\leq & e^{-\delta n\tau}\int e^{(\alpha+\delta) r^n_{R}}d\mu_\phi\\
&\leq & e^{-\delta n\tau}e^{n (\Psi_{R}(\alpha+\delta)+\eps)}\\
&=& o(e^{n \Psi_{R}(\alpha)})
      \end{eqnarray*} 
\end{proof}

\section{Existence of inner and outer approximations, their properties and consequences of their equality}\label{sec-exis-psiA}

In the first subsection we prove a monotonicity result about cumulant generating functions.
The idea is to approximate the set $A$ from the inside and from the outside by finite unions of cylinders, and show that the inner cumulant generating function $\Psiin$ and the outer cumulant generating function $\Psiout$ exist. 
Finally, we study the consequence of their equality on the cumulant generating function and the rate function for the set $A$.


\subsection{Monotonicity of the cumulant generating function on rectangles} 

For $m$ an integer, let $\CB_m$ be the biggest union of $m$-cylinders contained in $A$
and  $\CC_m$ be the smallest union of $m$-cylinders which contains $A$.  
Then, we denote by $\CD_m$ the set $\CC_m\setminus\CB_m$ (See Figure~\ref{fig-bmacm}).
As any $(-m,m)$-cylinder is a union of $(-m-1,m+1)$-cylinders, we have $\CB_m\subset\CB_{m+1}\subset A\subset\CC_{m+1}\subset\CC_m$; Therefore $(\CD_m)$ is a decreasing sequence of compact set which converges to $\pA$. 

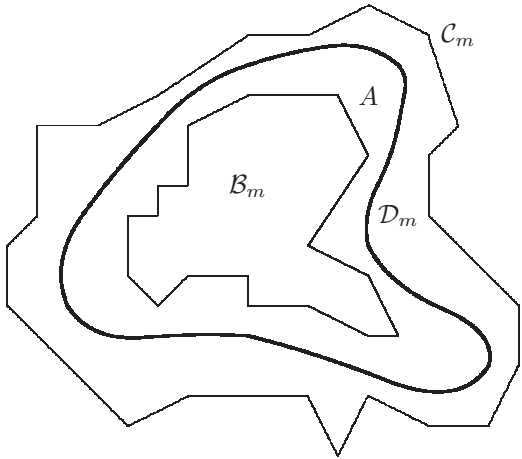
\begin{figure}

\unitlength 0.8 mm
\begin{picture}(95,80)(0,0)
\linethickness{0.3mm}
\qbezier(50,70)(45.56,68.63)(41.82,66.07)
\qbezier(41.82,66.07)(38.07,63.51)(35,60)
\qbezier(35,60)(28.59,53.37)(22.77,45.43)
\qbezier(22.77,45.43)(16.95,37.49)(20,30)
\qbezier(20,30)(23.87,24)(32.8,24.74)
\qbezier(32.8,24.74)(41.73,25.47)(50,25)
\qbezier(50,25)(61.02,22.32)(73.11,17.56)
\qbezier(73.11,17.56)(85.21,12.8)(90,20)
\qbezier(90,20)(91.51,25.52)(82.4,29.46)
\qbezier(82.4,29.46)(73.3,33.4)(70,40)
\qbezier(70,40)(69,45.03)(71.47,49.89)
\qbezier(71.47,49.89)(73.94,54.75)(75,60)
\qbezier(75,60)(75.5,62.65)(76.04,65.4)
\qbezier(76.04,65.4)(76.59,68.15)(75,70)
\qbezier(75,70)(70.34,74.1)(63.37,73.14)
\qbezier(63.37,73.14)(56.41,72.19)(50,70)
\linethickness{0.1mm}
\multiput(25,60)(0.24,0.12){42}{\line(1,0){0.24}}
\multiput(35,65)(0.18,0.12){83}{\line(1,0){0.18}}
\put(50,75){\line(1,0){10}}
\multiput(60,75)(0.24,0.12){42}{\line(1,0){0.24}}
\multiput(70,80)(0.24,-0.12){42}{\line(1,0){0.24}}
\multiput(80,75)(0.12,-0.36){42}{\line(0,-1){0.36}}
\multiput(80,55)(0.12,0.12){42}{\line(1,0){0.12}}
\put(80,45){\line(0,1){10}}
\multiput(80,45)(0.12,-0.12){83}{\line(1,0){0.12}}
\multiput(90,35)(0.12,-0.12){42}{\line(1,0){0.12}}
\put(95,20){\line(0,1){10}}
\multiput(90,10)(0.12,0.24){42}{\line(0,1){0.24}}
\put(80,10){\line(1,0){10}}
\multiput(70,15)(0.24,-0.12){42}{\line(1,0){0.24}}
\multiput(65,5)(0.12,0.24){42}{\line(0,1){0.24}}
\multiput(60,15)(0.12,-0.24){42}{\line(0,-1){0.24}}
\put(45,15){\line(1,0){15}}
\put(40,15){\line(1,0){5}}
\multiput(30,10)(0.24,0.12){42}{\line(1,0){0.24}}
\multiput(25,15)(0.12,-0.12){42}{\line(1,0){0.12}}
\multiput(15,25)(0.12,-0.12){83}{\line(1,0){0.12}}
\multiput(10,30)(0.12,-0.12){42}{\line(1,0){0.12}}
\put(10,30){\line(0,1){10}}
\multiput(10,40)(0.12,0.12){42}{\line(1,0){0.12}}
\put(15,45){\line(0,1){15}}
\put(15,60){\line(1,0){10}}
\linethickness{0.1mm}
\put(30,40){\line(0,1){5}}
\put(30,45){\line(1,0){5}}
\put(35,45){\line(0,1){5}}
\put(35,50){\line(1,0){5}}
\put(40,50){\line(0,1){10}}
\multiput(40,60)(0.24,0.12){42}{\line(1,0){0.24}}
\qbezier(50,65)(50.95,65.01)(57.55,65.03)
\qbezier(57.55,65.03)(64.14,65.05)(65,65)
\qbezier(65,65)(65.24,64.33)(67.47,59.98)
\qbezier(67.47,59.98)(69.7,55.63)(70,55)
\qbezier(70,55)(69.36,54.05)(64.99,47.5)
\qbezier(64.99,47.5)(60.63,40.94)(60,40)
\qbezier(60,40)(60.62,39.69)(65,37.5)
\qbezier(65,37.5)(69.38,35.31)(70,35)
\qbezier(70,35)(70.31,34.37)(72.5,30)
\qbezier(72.5,30)(74.69,25.62)(75,25)
\qbezier(75,25)(74.69,25)(72.5,25)
\qbezier(72.5,25)(70.31,25)(70,25)
\qbezier(70,25)(69.37,25.31)(65,27.5)
\qbezier(65,27.5)(60.63,29.69)(60,30)
\qbezier(60,30)(59.37,30)(55,30)
\qbezier(55,30)(50.62,30)(50,30)
\put(50,30){\line(0,1){5}}
\put(40,35){\line(1,0){10}}
\multiput(35,30)(0.12,0.12){42}{\line(1,0){0.12}}
\multiput(30,35)(0.12,-0.12){42}{\line(1,0){0.12}}
\put(30,35){\line(0,1){5}}
\put(50,50){\makebox(0,0)[cc]{$\CB_m$}}

\put(70,65){\makebox(0,0)[cc]{$A$}}

\put(85,75){\makebox(0,0)[cc]{$\CC_m$}}

\put(75,45){\makebox(0,0)[cc]{$\CD_m$}}

\end{picture}
\caption{Inner and outer approximation of the set $A$ by $m$-cylinders.}\label{fig-bmacm}
\end{figure}

Following what is done above, there exists two analytic functions $\Psi_{\CB_m}$ and $\Psi_{\CC_m}$ respectively defined on $]-\8,\alpha(\CB_m)[$ and $]-\8,\alpha(\CC_m)[$. 
As it is said above, $\alpha(\CB_m)$ is the difference between $\CP_\phi(\S)$ and the topological $\phi$-pressure of the dotted system $\S_{\CB_m}$. In the same way, $\alpha(\CC_m)$ is the difference between $\CP_\phi(\S)$ and the topological $\phi$-pressure of the dotted system $\S_{\CC_m}$. Now, we clearly have $\S_{\CC_m}\subset\S_{\CB_m}$, because $\CC_m\supset\CB_m$. We also have $\CC_{m+1}\subset\CC_m$ and $\CB_m\subset\CB_{m+1}$. Therefore the sequence $(\alpha(\CB_m))_m$ is non-decreasing and the sequence $(\alpha(\CC_m))_m$ is non-increasing. Moreover, for any $m$, 
$$\alpha(\CB_m)\leq \alpha(\CC_m).$$
The sequence $(\alpha(\CB_m))_m$ is thus converging to some limit $\alphain\in(0,+\8]$. Hence, for any $\alpha<\alphain$, and for any sufficiently large $m$, the functions $\Psi_{\CB_m}$ and $\Psi_{\CC_m}$ are real-analytic on $]-\8,\alpha[$. 

Let us define the lower and upper cumulant generating functions of $A$ by
\[
\ul\Psi_A(\alpha)=\liminf_{n\to\infty}\frac1n \log\int e^{\alpha r_A^n}d\mu_\phi
\quad\text{and}\quad
\ol\Psi_A(\alpha)=\limsup_{n\to\infty}\frac1n \log\int e^{\alpha r_A^n}d\mu_\phi.
\]

\begin{proposition}\label{prop-inegal-gd}
For any $0\leq \alpha<\alphain$ and for any sufficiently large $m$, we have 
$$\Psi_{\CC_m}(\alpha)\leq \ul\Psi_A(\alpha)\leq\ol\Psi_A(\alpha)\leq \Psi_{\CB_m}(\alpha).$$
For any $\alpha<0$ we have 
$$\Psi_{\CB_m}(\alpha)\leq \ul\Psi_A(\alpha)\leq\ol\Psi_A(\alpha)\leq \Psi_{\CC_m}(\alpha).$$
\end{proposition}
\begin{proof}
The double inclusion $\CB_m\subset A\subset\CC_m$ implies that  $r^n_{\CB_m}\geq r^n_A\geq r^n_{\CC_m}$. The result follows then immediately by integration and taking the appropriate limits.
\end{proof}

\begin{remark}\label{rem1-ineg-psi}
Forgetting the set $A$ in the previous proof, we have in fact proved that if $\CB$ and $\CC$ are finite unions of cylinders satisfying $\CB\subset\CC$, then  for any $\alpha<0$,
\begin{equation}\label{equ3-inega-psi}
  \Psi_\CB(\alpha)\leq \Psi_\CC(\alpha),
\end{equation}
and for any $\alpha\geq 0$ (but sufficiently small such that the functions are well-defined)
\begin{equation}\label{equ4-inega-psi}
 \Psi_\CB(\alpha)\geq \Psi_\CC(\alpha).
\end{equation}
\end{remark}

\subsection{Existence of inner and outer cumulant generating functions}

Let us pick some $0<\alpha<\alphain$. By  (\ref{equ4-inega-psi}), the sequence of functions $(\Psi_{\CB_m})_m$ is a non-increasing sequence of non-decreasing convex functions on $[0,\alpha[$ (for sufficiently large $m$ !). It thus (simply) converges to some limit function ${\Psiin}$. This function ${\Psiin}$ has to be convex, thus continuous on $]0,\alpha[$. It also has to be non-decreasing, thus it must be continuous on $[0,\alpha[$. Moreover the Dini Theorem yields that the convergence is uniform on every compact set included in $[0,\alpha[$. This occurs for any $0<\alpha<\alphain$, thus the limit function ${\Psiin}$ is non-decreasing and continuous on $[0,\alphain[$ and the convergence is uniform on every compact set included in $[0,\alphain[$.

In the same way the sequence of functions $(\Psi_{\CC_m})_m$ is a non-decreasing sequence of non-decreasing convex functions on $[0,\alphain[$ (for any $m$ !). It thus (simply) converges to some limit function ${\Psiout}$; this function ${\Psiout}$ is convex and continuous on $]0,\alphain[$. Note that by (\ref{equ4-inega-psi}) we have 
$$0\leq {\Psiout}\leq {\Psiin}.$$
As $\Psi_{\CB_m}(0)=\Psi_{\CC_m}(0)=0$ for any $m$, ${\Psiout}$ is continuous on $[0,\alphain[$, and the convergence is uniform. 

We do the same work on $]-\8,0]$ using (\ref{equ3-inega-psi}) instead of (\ref{equ4-inega-psi}). Note that for $\alpha\leq 0$ we have 
$$0\geq {\Psiout}(\alpha)\geq {\Psiin}(\alpha).$$
The two functions ${\Psiin}$ and ${\Psiout}$ are convex and non-decreasing.  By Proposition~\ref{prop-inegal-gd} we have 
$${\Psiout}(\alpha)\leq \ul\Psi_A(\alpha)\leq \ol\Psi_A(\alpha)\leq {\Psiin}(\alpha) \mbox{ for }\alpha\geq 0,$$
$$\mbox{and }{\Psiin}(\alpha)\leq \ul\Psi_A(\alpha)\leq \ol\Psi_A(\alpha)\leq {\Psiout}(\alpha) \mbox{ for }\alpha\leq 0.$$

We emphasize that the existence of the limit $\Psi_A(\alpha)$ immediately follows from the equality $\Psiin(\alpha)=\Psiout(\alpha)$. Moreover, it also implies, quite surprisingly, despite any knowledge about the  differentiability of the function $\Psi_A$, that $\Phi_A$ exists and is the Legendre transform of $\Psi_A$:

\begin{proposition}\label{pro-inegalout}
 If $\Psiin=\Psiout$ on an open interval $(\ul\alpha,\ol\alpha)$ then 
 
 (i) the cumulant generating function $\Psi_A$ exists, is equal to $\Psiin=\Psiout$ on this interval and it is a convex, continuous, non-decreasing function. In addition $\Psi_A$ is differentiable for all but countably many points and left and right limits of $\Psi_A'$ exists everywhere on the closure of the interval.
 
 (ii)  for all $u\in(\Psi_A'(\ul\alpha+),\Psi_A'(\ol\alpha-))$, the rate function $\Phi_A(u)$ exists and satisfies the relation 
 \[
 \Phi_A(u)=\inf_{\alpha\in(\ul\alpha,\ol\alpha)}\{\Psi_A(\alpha)-\alpha u\}.
 \]

(iii) if $\ul\alpha=-\infty$ then for any $u<\Psi_A'(\ol\alpha)$, except possibly for $u=\Psi_A'(-\infty)$, the rate function $\Phi_A(u)$ exists and satisfies the relation 
\[
 \Phi_A(u)=\inf_{\alpha<\ol\alpha}\{\Psi_A(\alpha)-\alpha u\}.
\]
\end{proposition}

\begin{proof}
(i) is straightforward.

(ii) Using exponential Markov inequality we immediately get the upper bound
\begin{equation}\label{eq:phiasup}
\begin{split}
\ol\Phi_A(u):=\limsup_{n\to\infty}\frac1n\log\mu_\phi(r_A^n\ge nu)&\le\inf_{\alpha}\Psi_A(\alpha)-\alpha u
\quad (u>\frac1\mu(A))\\
\ol\Phi_A(u):=\limsup_{n\to\infty}\frac1n\log\mu_\phi(r_A^n\le nu)&\le\inf_{\alpha}\Psi_A(\alpha)-\alpha u
\quad (u<\frac1\mu(A)).
\end{split}
\end{equation}
We now prove that for $\ul\Phi_A$ defined with the $\liminf$ the lower bound also holds true.
Fix $u\in (\Psi_A'(\ul\alpha+),\Psi_A'(\ol\alpha-))$. By convexity of $\Psi_A$ the function $\Psi_A(\alpha)-\alpha u$ attains its infimum for some 
$\alpha_*\in(\ul\alpha,\ol\alpha)$, and the limits at the endpoints of the interval are strictly larger.
Let $\eps>0$ so small that there exists $\ul\beta<\alpha_*<\ol\beta$ in the interval $(\ul\alpha,\ol\alpha)$ such that for $\alpha=\ul\beta$ and $\ol\beta$ we have $\Psi_A(\alpha)-\alpha u>\Psi_A(\alpha_*)-\alpha_* u+2\eps$.

Suppose $\alpha_*>0$, i.e. $u>\frac1{\mu(A)}$. By equality of the outer approximation $\Psiout$ with $\Psi_A$, there exists a set $\CC$ which is a finite union of cylinders such that $A\subset \CC$ and $\Psi_{\CC}\le\Psi_A\le\Psi_\CC+\eps$ on the interval $[0,\ol\beta]$. This implies that $\Psi_\CC(\ol\beta)-\ol\beta u>\Psi_\CC(\alpha_*)-\alpha_* u$, therefore by convexity of $\Psi_\CC(\alpha)-\alpha u$ we have
\begin{equation}\label{eq:infpsi}
\inf_\alpha\Psi_\CC(\alpha)-\alpha u = 
\inf_{\alpha\in[0,\ol\beta]}\Psi_\CC(\alpha)-\alpha u
\ge \inf_\alpha\Psi_A(\alpha)-\alpha u-\eps.
\end{equation}
On the other hand $r_\CC^n\ge nu$ implies $r_A^n\ge nu$, hence $\ul\Phi_A(u)\ge\Phi_\CC(u)$. Since the large deviations principle holds for the set $\CC$ (see Remark~\ref{rem-ldprectangle}), we have $\Phi_\CC(u)=\inf_\alpha\Psi_\CC(\alpha)-\alpha$ and the conclusion follows from inequality \eqref{eq:infpsi}.
The case $\alpha_*<0$ can be treated in the same way by considering the inner approximation $\Psiin$.

(iii) By (ii) it is enough to consider $u<\lim_{\alpha\to-\infty}\Psi_A'(\alpha)$, but in this case $\inf_\alpha\Psi_A(u)-\alpha u=-\infty$ which implies  the result by \eqref{eq:phiasup}.
\end{proof}

We remark that the exceptional value $\Psi'(-\infty)$ in statement (iii) is the slope of $\Psi_A$ at $-\infty$ which is $\frac{1}{\max(A)}$. At this point the value of $\Phi_A$ jumps from $-\infty$ to finite values. 

Note that if $\lim_{\alpha\to\ol\alpha}\Psi_A'(\alpha)=+\infty$ in Proposition~\ref{pro-inegalout} then
 one gets the Large Deviations Principle and the formula~\eqref{equ-psidonnephi} holds for every $u>\Psi_A'(\ul\alpha+)$.
In Sections~\ref{sec-proof-th1} we will prove under the assumption in statement~\ref{point1} of the theorem the existence of $\Psi_A$ on the interval $(-\infty,\alphain)$. However, we are not able to prove that the limit of the derivative at $\alphain$ is infinite. We doubt that it could be finite when $\alphain<+\infty$.
Still, this is clearly the case whenever $\alphain=+\infty$, which is equivalent to $\min(A)>0$.

Nevertheless, using next proposition which exploits the symmetry of our assumption (since $A$ and $A^c$ share the same boundary), this will be sufficient to get the existence of the rate function $\Phi_A$ on the whole interval (except at the discontinuity).

\begin{proposition}\label{pro-psihalf}
Assume that the Large Deviations Principle of return times into $A^c$ holds for any $u\le\frac{1}{\mu(A^c)}$  with a continuous rate function $\Phi_{A^c}$ (except possibly for $u=\frac{1}{\max(A^c)}$). Then the Large Deviations Principle for return times into $A$ holds for any $u\ge\frac{1}{\mu(A)}$ (except possibly for $u=\frac{1}{\min(A)}$) with a rate function $\Phi_A$ which satisfies
\[
\Phi_A(u)=\left(u-1\right) \Phi_{A^c}\left(\frac{u}{u-1}\right).
\]
In particular, if $\Phi_{A^c}$ is the Legendre transform of the cumulant generating function $\Psi_{A^c}$ then
\[
\Phi_A(u)=\inf_{\alpha<0} \left\{-\alpha u +(u-1)\Psi_{A^c}\left(\alpha\right)\right\}.
\]
\end{proposition}

\begin{proof}
Observe that if $u>\frac{1}{\mu(A)}\ge1$ then $r_A^n\ge nu$ 
if and only if the orbit entered at most $n$ times in $A$ before the time $\lfloor nu\rfloor$,
which means that the orbit entered at least $\lfloor n(u-1)\rfloor$ times into $A^c$ before the time $\lfloor nu\rfloor$. Therefore
\[
\frac1n\log\mu_\phi(r_A^n\ge nu) = \frac1n\log\mu_\phi\left(r_{A^c}^{\lfloor n(u-1)\rfloor}< {\lfloor n(u-1)\rfloor}\frac{\lfloor nu\rfloor}{\lfloor n(u-1)\rfloor}\right)
\]
and the result follows by taking the limit as $n\to\infty$.

\end{proof}


\section{Coincidence of inner and outer approximation in the case of a small pressure boundary}\label{sec-proof-th2}

The goal of this section is to prove the statement~\ref{point2} of the theorem.
By the previous analysis (See Proposition~\ref{pro-inegalout}) it is sufficient to prove the existence of some interval $(\ul\alpha,\ol\alpha)\ni0$ 
such that for any $\alpha\in(\ul\alpha,\ol\alpha)$ we have $\Psiin(\alpha)=\Psiout(\alpha)$.

\subsection{A more explicit condition to get a small pressure boundary}

Let $K\subset\Sigma$ be a Borel set. We recall our definition of its $\phi$-pressure:
\[
\wt\CP_\phi(K)=\sup \left\{ h_\nu+\int \phi d\nu\colon \nu\text{ ergodic and }\nu(K)>0\right\}.
\]
Note that it is not the same as the one defined as a dimension like characteristic with forward cylinders. That would satisfy us in the case of expanding maps, but for diffeomorphisms it would lead to a condition much too strong. 
\begin{proposition}\label{pro-bdpress}
Let $K\subset\Sigma$ be a Borel set and let $V_n(K)$ be the smallest union of $(-n,n)$-cylinders which contains $K$.
If there exist some constants $c>0$ and $\theta>0$ such that $\mu_\phi( V_n(K) ) \le c e^{-\theta n}$ for all integers $n$, then $\wt\CP_\phi(K)\le \CP_\phi(\Sigma)-\frac12\theta$.
\end{proposition}
\begin{proof}
Let $\wt S_n\phi(x)=\sum_{k=-n}^{n-1}\phi(x)$ denote the two sided Birkhoff sum and $C_{-n,n}(x)$ the $(-n,n)$-cylinder containing the point $x\in\Sigma$.
Recall that since $\mu_\phi$ is a Gibbs measure, for some constant $b>0$ and for every $x\in\Sigma$ we have
\[
e^{-b} \le \frac{\mu_\phi(C_{-n,n}(x))}{\exp\left(\wt S_n\phi(x)-2n\CP_\phi(\Sigma)\right)} \le e^b.
\]
Let $\nu$ be an ergodic measure such that $\nu(K)>0$. We have
\[
\begin{split}
c &\ge e^{\theta n} \mu_\phi(V_n(K)) \\
& \ge 
\int_K \exp\left(\theta n+\log\frac{\mu_\phi(C_{-n,n}(x))}{\nu(C_{-n,n}(x))}\right)d\nu(x) \\
&=\int_K \exp\left( 2n\left[\frac\theta2+\CP_\phi(\Sigma)+\frac1{2n}\wt S_n\phi(x)-\frac1{2n}\log\nu(C_{-n,n}(x) )\right]\right) d\nu(x).
\end{split}
\]
The Shannon-McMillan-Breiman theorem and the ergodic theorem implies the convergence $\nu$-a.e. of the term into square bracket to the value 
\[
\frac\theta2+\CP_\phi(\Sigma)+h_\nu+\int\phi d\nu,
\]
which cannot be positive according to Fatou's lemma.
\end{proof}

\begin{proof}[Proof of Proposition~\ref{pro-bdpress-mfd}]
Take a Markov partition of sufficiently small diameter and denote by $\pi\colon \Sigma\to M$ the semi-conjugacy. We know that the diameter of the image by $\pi$ of a $(-n,n)$-cylinder goes uniformly to zero at an exponential rate. Thus, setting $A=\pi^{-1}V$, since $\pA\subset \pi^{-1}\partial V$, we get that the $(-n,n)$-cylindrical neighborhood of $\pA$ has a measure exponentially small, therefore Proposition~\ref{pro-bdpress} applies.
\end{proof}
\subsection{Coincidence for positive values of $\alpha$}

Lemma \ref{lem-carac-mus} characterizes $\Psi_{\CB_m}$ and $\Psi_{\CC_m}$: as soon as $\Psi_{\CB_m}(\alpha)$ is defined, $\Psi_{\CB_m}(\alpha)$ is the unique real number $t=t(\alpha, m)$ such that the topological pressure associated to the potential $\phi-t\BBone_{\CB_m}$, $\CP_{\phi-t\BBone_{\CB_m}}$, equals $\CP_\phi(\S)-\alpha$.

Similarly, $\Psi_{\CC_m}(\alpha)$ is the unique real number $t=t(\alpha, m)$ such that the topological pressure associated to the potential $\phi-t\BBone_{\CC_m}$, $\CP_{\phi-t\BBone_{\CC_m}}$, equals $\CP_\phi(\S)-\alpha$.

Let us pick some $\alpha>0$. We denote by $m_{\CB_m,\alpha}$ the measure $m_S$ obtained when we have $R=\CB_m$ and $S=\CP_{\phi}(\S)-\alpha$ in subsection \ref{subsec-thermo-loc}. This measure is the unique equilibrium state associated to the potential $\phi-\Psi_{\CB_m}(\alpha) \BBone_{\CB_m}$. The measure weights $\CB_m$, hence $\CC_m$ and we can take the induced measure on $\CC_m$. Therefore we have (omitting the subscribe $m$ for convenience)
\begin{eqnarray*}
h_{m_{\CB,\alpha}}(f)+\int\phi-\Psi_\CB(\alpha)\,dm_{\CB,\alpha}&=&\CP_\phi(\S)-\alpha\\
h_{m_{\CB,\alpha}}(f)+\int\phi-(\CP_\phi(\S)-\alpha)\,dm_{\CB,\alpha}&=&m_{\CB,\alpha}(\CB)\Psi_\CB(\alpha)\\
m_{\CB,\alpha}(\CC)\left(h_{\mu_{\CB,\alpha,\CC}}(g_\CC)+\int S_{r_C}(\phi-(\CP_\phi(\S)-\alpha))d\mu_{\CB,\alpha,\CC}\right)&=&m_{\CB,\alpha}(\CB)\Psi_\CB(\alpha),
\end{eqnarray*}
where $\mu_{\CB,\alpha,\CC}$ is the conditional measure $m_{\CB,\alpha}{}_{|\CC}$
and $g_\CC$ is the first return map on $\CC$.
This measure has a pressure in $\CC$ lower than $\Psi_\CC(\alpha)$; we thus get 
\begin{equation}\label{equ-mino-psisss}
\frac{m_{\CB_m,\alpha}(\CB_m)}{m_{\CB_m,\alpha}(\CC_m)}\Psi_{\CB_m}(\alpha)\leq \Psi_{\CC_m}(\alpha).
\end{equation}

Recall that for positive $\alpha$, $0<\Psi_{\CC_m}(\alpha)\le\Psi_{\CB_m}(\alpha)$ and are upper bounded (uniformly in every compact set in $[0,\alphain[$). 

\begin{proposition}\label{prop-keyalpha1}
There exists some $\ol\alpha>0$ such that for every $\alpha\in(0,\ol\alpha)$,
$$\lim_{m\rightarrow\pinf}\frac{m_{\CB_m,\alpha}(\CD_m)}{m_{\CB_m,\alpha}(\CC_m)}=0.$$
In particular, $\Psiin=\Psiout$ on this interval.
\end{proposition}

The proof of the proposition is an immediate consequence of these two lemmas and Inequality~\eqref{equ-mino-psisss}.
\begin{lemma}\label{lem-meslimcposi}
For any $\alpha\in(0,\alphain)$ we have $\disp\liminf_{m\rightarrow\pinf}{m_{\CB_m,\alpha}(\CC_m)}>0$
\end{lemma}
\begin{proof}
Let us denote by $\mu_m$ the measure $m_{\CB_m,\alpha}$, and pick any accumulation point $\nu$ of $(\mu_m)$ such that $\mu_m(\CB_m)$ converges (up to the correct subsequence) to $L:=\disp\liminf_{m\rightarrow\pinf}{\mu_m(\CB_m)}$. Let us show that $L>0$ whenever $\alpha<\alphain$.

Since $\mu_m$ is an equilibrium state we have 
\begin{equation}
h_{\mu_m}+\int \phi\,d\mu_m-\Psi_{\CB_m}(\alpha)\mu_m(\CB_m)=\CP_\phi(\S)-\alpha.
\label{equ-press}
\end{equation}
By semi-continuity for the metric entropy and the continuity of $\phi$ 
we obtain 
\begin{equation}
\CP_\nu(\phi)=h_\nu+\int\phi d\nu
\geq \CP_\phi(\S)-\alpha+\Psiin(\alpha)L.\label{equ-minopress}
\end{equation}
If $\nu(\CB_j)=0$ then this yields that $\nu$ is a $\s$-invariant measure for the dotted system $\S_{\CB_j}$. Hence, its $\phi$-pressure must be smaller than $\CP_\phi(\S_{\CB_j})$, which is by definition $\CP_\phi(\S)-\alpha(\CB_j)$. If this holds for every $j$ then $\CP_\nu(\phi)\leq \CP_\phi(\S)-\alphain$ (remember that $(\alpha(\CB_j))$ converges to $\alphain$). 
On the other hand by \eqref{equ-minopress} we had $\CP_\nu(\phi)\geq \CP_\phi(\S)-\alpha$, and $\alpha<\alphain$. This yields a contradiction. Therefore $\nu(\CB_j)>0$ for some $j$. Additionally, whenever $m\ge j$ we get $\mu_m(\CB_m)\ge\mu_m(\CB_j)$, and the later converges to $\nu(\CB_j)$ by continuity of $\BBone_{\CB_j}$.
This achieves the proof of the lemma since $\CC_m\supset \CB_m$ for any $m$.
\end{proof}

\begin{lemma}\label{lem-mesbordnul}
Let us set $\ol\alpha:=\min(\CP_\phi(\S)-\wt\CP_\phi(\pA),\alphain)>0$. For every $\alpha\in(0,\ol\alpha)$ we have $\disp\lim_{m\rightarrow\pinf}{m_{\CB_m,\alpha}(\CD_m)}=0$.
\end{lemma}
\begin{proof}
Let us fix some $\alpha\in(0,\ol\alpha)$.
Let us pick any accumulation point $\nu$ for the sequence of measures $\mu_m$ (we keep the notation of the preceding lemma). We claim that $\nu(\pA)=0$.

Assume for a contradiction that $\nu(\pA)>0$. Then let $H=\cup_{n\in\mathbb{Z}}f^{-n}\pA$ be the invariant hull of $\pA$.
Let $\nu_0$ and $\nu_1$ be the conditional measures of $\nu$ on $H$ and $\Sigma\setminus H$.
These two invariant probabilities are such that $\nu=p\nu_0+q\nu_1$ for some $p>0$.
Observe that by definition, any ergodic component of $\nu_0$ gives mass to $H$. Therefore even if $\nu_0$ is not ergodic, since the entropy is affine we still get that  
\begin{equation}\label{equ2-bordpetitepression}
 h_{\nu_0}+\int\phi\,d\nu_0\leq \wt\CP_\phi(\pA) = \CP_\phi(\S)-\ol\alpha.
\end{equation}

Copying the equality (\ref{equ-press}), we get for every integers $m\ge j$ that
\[
h_{\mu_m}+\int \phi d\mu_m -\Psi_{\CB_m}(\alpha)\mu_m(\CB_j) \ge \CP_\phi(\Sigma)-\alpha.
\]
Thus letting $m\to\infty$ gives, since the entropy is semi-continuous and affine, 
\begin{equation}\label{equ3-bordpetitepression}
p\left(h_{\nu_0}+\int\phi\,d\nu_0-\Psiin(\alpha)\nu_0(\CB_j)\right)+q\left(h_{\nu_1}+\int\phi\,d\nu_1-\Psiin(\alpha)\nu_1(\CB_j)\right)\geq \CP_\phi(\S)-\alpha,
\end{equation}
Hence (\ref{equ2-bordpetitepression}) and (\ref{equ3-bordpetitepression}) yield that for every $j$
$$h_{\nu_1}+\int\phi\,d\nu_1-\Psiin(\alpha)\nu_1(\CB_j)\geq \CP_\phi(\S)-\alpha+\frac{p}{q}(\ol\alpha-\alpha).$$
We now choose $j$ large enough such that 
$$h_{\nu_1}+\int\phi\,d\nu_1-\Psi_{\CB_j}(\alpha)\nu_1(\CB_j)> \CP_\phi(\S)-\alpha$$
holds. This is a contradiction because the measure $\nu_1$ would have a $\phi-\Psi_{\CB_j}\BBone_{\CB_j}$-pressure strictly larger than the associated equilibrium state. Thus we have $\nu(\pA)=0$.

To finish the proof let us fix some $\eps>0$ and consider any $j$ such  that $\nu(\CD_j)<\eps$. Such an integer $j$ exists by outer regularity of the measure $\nu$ and because $\pA=\disp\bigcap{\downarrow}\CD_n$. Note that $\BBone_{\CD_j}$ is continuous. Now, for any $m\geq j$ we have $\CD_m\subset\CD_j$, and then we get
$$0\leq \limsup_m \mu_m(\CD_m)\leq \nu(\CD_j)<\eps.$$
This holds for every positive $\eps$, which proves the lemma.
\end{proof}

\begin{remark}\label{rem-alpha1}
We remark that under the assumption in statement~\ref{point1} of the theorem, we always have $\nu(A)=0$ for the measure $\nu$ constructed in Lemma~\ref{lem-mesbordnul},
therefore $\ol\alpha=\alphain$.
\end{remark}

\subsection{Coincidence for negative values of $\alpha$}

We remark that the measure $m_{\CC,\alpha}$ is a Gibbs measure with full topological support, thus it gives weight to $\CB$. Therefore we can copy the case $\alpha$ positive  and induce on $\CB$ (instead of $\CC$); we get similarly
\begin{equation}\label{equ-mino-psisss2}
\frac{m_{\CC_m,\alpha}(\CC_m)}{m_{\CC_m,\alpha}(\CB_m)}\Psi_{\CC_m}(\alpha)\leq \Psi_{\CB_m}(\alpha).
\end{equation}

\begin{proposition}\label{prop-keyalpha12}
There exists some real $\ul\alpha<0$ such that for every $\alpha\in(\ul\alpha,0)$,
$$\lim_{m\rightarrow\pinf}\frac{m_{\CC_m,\alpha}(\CD_m)}{m_{\CC_m,\alpha}(\CC_m)}=0.$$
In particular, $\Psiin=\Psiout$ on this interval.
\end{proposition}

The proof of the proposition is an immediate consequence of these two lemmas and Inequality~\eqref{equ-mino-psisss2}.

\begin{lemma}\label{lem-meslimcposi2}
For any negative $\alpha$ we have $\disp\liminf_{m\rightarrow\pinf}{m_{\CC_m,\alpha}(\CC_m)}>0$.
\end{lemma}
\begin{proof}
Let us denote by $\mu_m$ the measure $m_{\CC_m,\alpha}$, and pick any accumulation point $\nu$ of $(\mu_m)$ such that $\mu_m(\CC_m)$ converges (up to the correct subsequence) to $L:=\disp\liminf_{m\rightarrow\pinf}{\mu_m(\CC_m)}$.

Since $\mu_m$ is an equilibrium state we have 
\begin{equation}
h_{\mu_m}+\int \phi\,d\mu_m-\Psi_{\CC_m}(\alpha)\mu_m(\CC_m)=\CP_\phi(\S)-\alpha.
\label{equ-press2}
\end{equation}
By semi-continuity for the metric entropy and the continuity of $\phi$ 
we obtain 
\begin{equation}
\CP_\nu(\phi)=h_\nu+\int\phi d\nu
\geq \CP_\phi(\S)-\alpha+\Psiout(\alpha)L.\label{equ-minopress2}
\end{equation}
Therefore $L\neq0$ otherwise the right hand side would be larger than the topological pressure of $\phi$.
\end{proof}

\begin{lemma}\label{lem-mesbordnul2}
There exists $\ul\alpha<0$ such that for any $\alpha\in(\ul\alpha,0)$ we have
$\displaystyle\lim_{m\to\infty} m_{\CC_m,\alpha}(\CD_m)=0$.
\end{lemma}
\begin{proof}
We keep the notation of the preceding lemma. Let $\nu$ be an accumulation point of $(\mu_m)$. We first show that $\nu(\pA)=0$.
By equality \eqref{equ-press2} we get that for any integers $m\ge j$, since $\CC_j\supset \CC_m$ and now $\Psi_{\CC_m}(\alpha)<0$, we have
\[
h_{\mu_m}+\int\phi d\mu_m -\Psi_{\CC_m}(\alpha)\mu_m(\CC_j) \ge \CP_\phi(\Sigma)-\alpha.
\]
Letting $m\to\infty$ gives, since the entropy is semi-continuous and $\BBone_{\CC_j}$ is continuous, that
\[
h_\nu+\int \phi d\nu -\Psiout(\alpha)\nu(\CC_j)\ge\CP_\phi(\Sigma)-\alpha.
\]
Assume for a contradiction that $\nu(\pA)>0$ and decompose $\nu=p\nu_0+q\nu_1$ as in the case $\alpha$ positive.
Let $\delta>0$. By definition of $\Psiout$, for any $j$ sufficiently large we have $-\Psi_{\CC_j}(\alpha)\nu_1(\CC_j)+\delta\ge -\Psiout(\alpha)\nu_1(\CC_j)$.
Since the entropy is affine, we get
\[
p\left(h_{\nu_0}+\int\phi d\nu_0-\Psiout(\alpha)\nu_0(\CC_j) \right)+q\left(h_{\nu_1}+\int\phi d\nu_1-\Psi_{\CC_j}(\alpha)\nu_1(\CC_j)+\delta\right)\ge \CP_\phi(\Sigma)-\alpha.
\]
This together with \eqref{equ2-bordpetitepression} gives
\[
q\left(h_{\nu_1}+\int \phi d\nu_1-\Psi_{\CC_j}(\alpha)\nu_1(\CC_j)+\delta\right)
\ge
\CP_\phi(\Sigma)-\alpha-p\left( \CP_\phi(\Sigma)-\ol\alpha-\Psiout(\alpha)\nu_0(C_j)  \right).
\]
Since the pressure $\CP_{\nu_1}(\phi-\Psi_{\CC_j}(\alpha)\BBone_{\CC_j})\le P_\phi(\Sigma)-\alpha$ this implies that 
\[
q (P_\phi(\Sigma)-\alpha+\delta)\ge \CP_\phi(\Sigma)-\alpha-p\left( \CP_\phi(\Sigma)-\ol\alpha-\Psiout(\alpha)\nu_0(C_j)  \right).
\]
By outer regularity of the measure $\nu_0$ we have $\nu_0(C_j)\to \nu_0(\ol A)\le \max(\ol A)$ as $j\to\infty$. Since $\delta$ is arbitrary this gives
$p(-\ol\alpha-\Psiout(\alpha)+\alpha)\ge0$, which is contradictory if $p>0$
 and $\alpha$ is small enough, since the function $\alpha\mapsto \alpha-\Psiout(\alpha)\max(\ol A)$ is continuous and vanishes for $\alpha=0$. Thus there exists $\ul\alpha<0$ such that if $\alpha\in]\ul\alpha,0[$ we have $\nu(\pA)=0$.

The conclusion of the lemma follows as in the positive case.
\end{proof}

\begin{remark}\label{rem-alpha2}
We remark that under the assumption in statement~\ref{point1} of the theorem, we always have $\nu(A)=0$ for the measure $\nu$ constructed in Lemma~\ref{lem-mesbordnul2}, 
therefore $\ul\alpha=-\infty$.
\end{remark}

\section{A dynamical proof of the coincidence of inner and outer approximation in the case of totally negligible boundary}\label{sec-proof-th1}

In this section we give an alternative and somewhat more direct proof of the statement~\ref{point1} in our theorem.
By Proposition~\ref{pro-inegalout} and Proposition~\ref{pro-psihalf} it suffices to show the equality $\Psiin=\Psiout$ on the interval $(-\infty,0)$ for the set $A$ and its complement $A^c$. The hypotheses on the boundary is completely symmetric if we replace $A$ by $A^c$, so it is sufficient to prove the equality on the interval $(-\infty,0)$ for the set $A$ only. However, we also prove that the equality holds some interval $(-\infty,\alphain)$ for some $\alphain>0$.
This in turn not only implies that the rate function $\Phi_A$ exists on the whole interval $[0,+\infty)$
(except at discontinuity points), but also shows that the formula \eqref{equ-psidonnephi} is satisfied on some interval $[0,\ol u)$ for some $\ol u>\frac{1}{\mu_\phi(A)}$.

\subsection{Infinite rate function for return times near the boundary}

Recall that  $\CD_m=\CC_m\setminus\CB_m$ is the $m$-cylindrical neighborhood of the boundary $\pA$.
For convenience, and for general computations, we remove the subscript "$m$" and just write $\CD$. Our aim is to show that, the probability that the successive return times into $\CD_m$ are small, is extremely small. We first prove a key lemma. 

\begin{lemma}\label{lem-lim-mesmax-D}
With the assumption on $\pA$, $\lim_{m\rightarrow+\8}\max(\CD_m)=0$.
\end{lemma}
\begin{proof}
Since $\CD_m$ is decreasing the limit $\rho:=\lim_{m\rightarrow+\8}\max(\CD_m)$ exists.
For any $m$ there exists some probability $\mu_m$ such that 
$$\mu_m(\CD_m)\geq \max(\CD_m)-\frac1{m}.$$
Let us pick any accumulation point $\mu$ for the sequence of probabilities $(\mu_m)$. Recall that the map $\BBone_{\CD_m}$ is continuous. Let us pick some integer $m$. For simplicity we write converging sequences instead of converging subsequences.
\[
\mu(\CD_m)=\lim_\ninf\mu_n(\CD_m)
\geq \liminf_\ninf \mu_n(\CD_n)
\geq \lim_\ninf\max(\CD_n)-\frac1{n}=\rho.
\]
By outer regularity of the measure $\mu$ this yields that $\rho\le\lim\mu(\CD_m)= \mu(\pA)=0$.
\end{proof}

\begin{proposition}\label{pro-phid-infini}
For every $v>0$, there exists some $M=M(v)$ such that for every $m\geq M$, $\Phi_{\CD_m}(v)=-\8$.
\end{proposition}
\begin{proof}
Let $v>0$. By Lemma~\ref{lem-lim-mesmax-D} we always can consider $m$ large enough such that 
$$\frac1{\mu_\phi(\CD)}>v.$$
Note that $\CD$ is a union of $(-m,m)$-cylinders. We thus can use the large deviations principle for $(r^k_\CD)$ (see Remark~\ref{rem-ldprectangle}) which gives
\begin{eqnarray}
\Phi_\CD(v)=\lim_{n\to \infty}\frac1{n}\log\mu_\phi\left\{\frac{r_\CD^{n}}{n}\leq
v\right\}&=&\inf_{\wt\alpha<\alpha'}\left\{-v\wt\alpha+\Psi_\CD(\wt\alpha)\right\},\label{equ1-esti-retour-D}
\end{eqnarray}
where $\alpha'=\alpha(\CD)>0$ (it thus depends on $m$). 

\begin{figure}
\begin{center}
\unitlength 0.6 mm
\begin{picture}(122.5,95)(0,0)
\linethickness{0.1mm}
\put(80,0){\line(0,1){90}}
\put(80,90){\vector(0,1){0.12}}
\linethickness{0.1mm}
\put(0,60){\line(1,0){120}}
\put(120,60){\vector(1,0){0.12}}
\linethickness{0.3mm}
\multiput(22.5,10)(1.51,1.32){65}{\multiput(0,0)(0.15,0.13){5}{\line(1,0){0.15}}}
\linethickness{0.3mm}
\qbezier(110,95)(102.19,79.35)(97.38,71.53)
\qbezier(97.38,71.53)(92.56,63.71)(90,62.5)
\qbezier(90,62.5)(87.49,61.21)(78.47,59.41)
\qbezier(78.47,59.41)(69.45,57.6)(52.5,55)
\qbezier(52.5,55)(35.62,52.41)(21.78,50)
\qbezier(21.78,50)(7.95,47.59)(-5,45)
\put(100,75){\makebox(0,0)[cc]{}}

\put(124.5,60.5){\makebox(0,0)[cc]{$\wt\alpha$}}

\put(80.5,94.5){\makebox(0,0)[cc]{$\Psi_\CD(\wt\alpha)$}}

\put(117.5,57.5){\makebox(0,0)[cc]{}}

\put(45,20.5){\makebox(0,0)[cc]{$\wt\alpha v$}}

\linethickness{0.3mm}
\put(25,20){\line(0,1){25}}
\put(25,45){\vector(0,1){0.12}}
\put(25,20){\vector(0,-1){0.12}}
\put(35,32.5){\makebox(0,0)[cc]{$\Phi_\CD(v)$}}
\end{picture}
\end{center}
\caption{The rate function $\Phi_\CD(v)$ is the maximal distance between $\wt\alpha v$ and $\Psi_\CD(\wt\alpha)$ on the negative axis.}\label{fig-phid}
\end{figure}
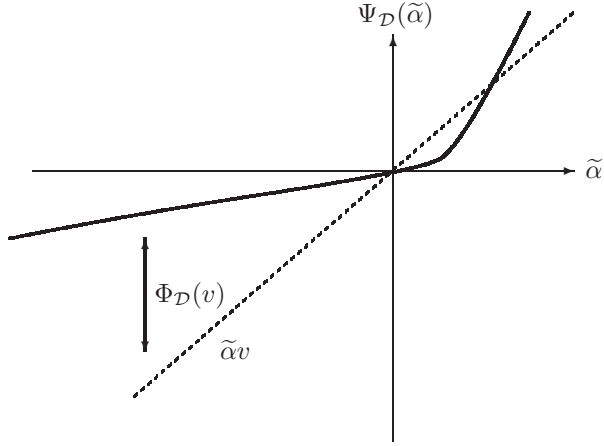

We emphasize that the slope of $\wt\alpha\mapsto \Psi_\CD(\wt\alpha)$ as $\wt\alpha$ goes to $-\8$ is $\frac1{\max(\CD)}$.
Lemma~\ref{lem-lim-mesmax-D} yields the existence of some $M=M(v)$ such that for every $m\geq M$, $\frac1{\max(\CD_m)}<v$. 
This implies by~\eqref{equ1-esti-retour-D} that $\Phi_{\CD_m}(v)\le \lim_{\wt\alpha\to-\8}\wt\alpha\left(v-\frac1{\max(D)}\right)=-\8$ (See Figure~\ref{fig-phid}).
\end{proof}

\subsection{Coincidence for positive values of $\alpha$}

Fix some $\alpha\in(0,\alphain)$ and $\delta$ such that $\alpha+\delta<\alphain$. For sufficiently large $m$, all the $\Psi_{\CB_m}$ are defined on $[0,\alpha+\delta]$ and equicontinuous. then choose a uniform $\tau$ in Proposition~\ref{prop-concen-mass} such that the mass concentration holds, namely
\begin{equation}\label{equ-mass-concen}
\Psi_{\CB_m}(\alpha)=\lim\frac{1}{n}\log \int_{r_{\CB_m}^n\le n\tau} e^{\alpha r_{\CB_m}^n}d\mu_\phi
\end{equation}
for all sufficiently large $m$.

Let us pick some fixed positive $\eps$. We have
\begin{equation}\label{equ1-majo-retou-eps}
\int_{\ r^n_B\leq n\tau} e^{\alpha r^n_\CB}\,d\mu_\phi\leq 
\int_{r^n_B\leq r^{n(1+\eps)}_\CC} e^{\alpha r^{n (1+\eps)}_\CC}\,d\mu_\phi
+\int_{r^{n(1+\eps)}_\CC<r^n_B\leq n\tau} e^{\alpha r^n_\CB}\,d\mu_\phi.
\end{equation}

The first term in the right hand side of this equation is simply bounded by
\begin{equation}\label{equ-rBpetit}
\int e^{\alpha r^{n (1+\eps)}_\CC}\,d\mu_\phi \le e^{n(1+2\eps)\Psi_{\CC}(\alpha)}
\end{equation}
provided $n$ is sufficiently large.

We turn to the second term. 
The condition  $r^{n(1+\eps)}_\CC<r^n_B\leq n \tau$ implies that $r_\CD^{n\eps}\leq n\tau$, hence we get
\begin{equation}
\int_{r^{n (1+\eps)}_\CC<r^n_B\leq n \tau} e^{\alpha r^n_\CB}\,d\mu_\phi
\leq e^{\alpha n \tau}\mu_\phi\left(r^{n \eps}_{\CD}\le n\tau\right)
= e^{\alpha n \tau}\mu_\phi(r^{n\eps}_\CD\leq (n \eps) \frac{\tau}{\eps}).\label{equ2-rab-majo-psib}
\end{equation}

By Proposition~\ref{pro-phid-infini}, if we consider $m\geq M(\frac\tau\eps)$ for some fixed $\eps$, for $n$ large enough we get
$$\mu_\phi(r^{n \eps}_\CD\leq n \eps\frac\tau\eps)\leq e^{-2\alpha \tau n},$$
 Therefore,  (\ref{equ2-rab-majo-psib}) gives for $n$ sufficiently large
$$\int_{r^{n (1+\eps)}_\CC<r^n_B\leq n \tau} e^{\alpha r^n_\CB}\,d\mu_\phi\leq 1.$$
Recall that $\Psi_\CC(\alpha)\geq 0$ for $\alpha\geq 0$. 
Then, \eqref{equ1-majo-retou-eps} together with \eqref{equ-rBpetit} and \eqref{equ-mass-concen} yield that
\[
\Psi_{\CB}(\alpha) \le (1+2\eps) \Psi_{\CC}(\alpha)
\]  
It follows from Proposition~\ref{prop-inegal-gd} that
\begin{equation}\label{equ-ine-psia}
\Psiin(\alpha)\leq (1+2\eps)\Psiout(\alpha).
\end{equation}
Letting $\eps$ go to $0$ we get that $\Psiin(\alpha)=\Psiout(\alpha)$.

\subsection{Coincidence for negative values of $\alpha$}

We now do the proof for a fixed $\alpha<0$. Here again we omit the subscript ``$m$'' when it is not necessary. We also pick some positive $\eps$. Then, we have: 

\begin{eqnarray}\int e^{\alpha r^n_\CB}\,d\mu_\phi
&\geq& \int_{r^n_B\leq r^{n (1+\eps)}_\CC} e^{\alpha r^{n (1+\eps)}_\CC}\,d\mu_\phi\nonumber\\
&\geq & \int e^{\alpha r^{n (1+\eps)}_\CC}\,d\mu_\phi-\int_{r^n_B> r^{n (1+\eps)}_\CC} e^{\alpha r^{n (1+\eps)}_\CC}\,d\mu_\phi.
\label{equ0-mino-gd-neg}
\end{eqnarray}
Let us pick some positive real $\wt\tau$ which will be chosen latter. We have 
\begin{eqnarray}
\int_{ r^n_B> r^{n (1+\eps)}_\CC} e^{\alpha r^{n (1+\eps)}_\CC}\,d\mu_\phi&\leq & \mu_\phi\left({r^n_\CB>r^{n (1+\eps)}_\CC>n (1+\eps)\wt\tau}\right)+\,\nonumber\\
&&\hskip 1cm +\mu_\phi\left({r^n_B> r^{n (1+\eps)}_\CC\cap n (1+\eps)\wt\tau\geq r^{n (1+\eps)}_\CC}\right)\nonumber\\
&\leq & \mu_\phi\left( r^{n (1+\eps)}_\CC>n (1+\eps)\wt\tau\right)
+\mu_\phi\left(r_\CD^{n\eps} \le n
 (1+\eps)\wt\tau \right).\nonumber\\
&&\label{equ1-mino-gd-neg}
\end{eqnarray}
The large deviations principle for $r^k_\CC$ means 
\begin{eqnarray}
\Phi_\CC(\wt\tau):=\lim_{n\to \infty}\frac1n\log\mu_\phi\left\{\frac{r_\CC^{n}}{n}\geq
\wt\tau\right\}&=& \inf_{\wt\alpha<\alpha'}\left\{-
\wt\tau\wt\alpha+\Psi_\CC(\wt\alpha)\right\}
\label{equ2-esti-retour-C}
\end{eqnarray}
for some $\alpha'>\alphain$. 
Fix some $j$ and some $\wt\alpha\in(0,\alpha(\CB_j))$. Choose then $\wt\tau$ such that
\[
-\wt\tau \wt\alpha+\Psi_{\CB_j}(\wt\alpha)<2 \Psi_{\CB_j}(\alpha)<0.
\]
Recall that on $\R_+$ all the $\Psi_\CC$ are lower than all the $\Psi_\CB$, and the converse holds on $\R_-$. 
Therefore we get for every $m$ that 
\begin{equation}\label{equ-esti-wttau}
-\wt\tau \wt\alpha+\Psi_{\CC_m}(\wt\alpha)<2 \Psi_{\CC_m}(\alpha)<0.
\end{equation}
For $n$ large enough, (\ref{equ2-esti-retour-C}) and (\ref{equ-esti-wttau}) yield 
\begin{equation}\label{equ3-mino-gd-neg}
\mu_\phi\left(r^{n (1+\eps)}_\CC>n (1+\eps)\wt\tau\right)\leq e^{n (1+\eps)(\Phi_\CC(\wt\tau)+\eps)}\leq e^{n (1+\eps)(2\Psi_\CC(\alpha)+\eps)}.
\end{equation}
Following Proposition~\ref{pro-phid-infini} we get 
\begin{equation}\label{equ4-mino-gd-neg}
\mu_\phi\left(r_\CD^{n\eps} \le n (1+\eps)\wt\tau \right)\leq e^{2n (1+\eps) \Psi_\CC(\alpha)}
\end{equation}
for every large enough $m$ and for every large enough $n$. Therefore (\ref{equ0-mino-gd-neg}),  (\ref{equ3-mino-gd-neg}),  and (\ref{equ4-mino-gd-neg}) yield for every large enough $m$ and for every large enough $n$:
$$\int e^{\alpha r^n_{\CB_m}}\,d\mu_\phi\geq e^{n (1+\eps) (\Psi_{\CC_m}(\alpha)-\eps)}-e^{2n (1+\eps) \Psi_{\CC_m}(\alpha)}-e^{n (1+\eps)(2\Psi_{\CC_m}(\alpha)+\eps)},$$
For fixed $m$, letting $n$ go to $+\8$ and using Proposition \ref{prop-inegal-gd} with $\alpha<0$ we get for every $\eps>0$
\begin{equation}\label{equ2-ine-psia}
\Psiin(\alpha)\geq (1+\eps)(\Psiout(\alpha)-\eps).
\end{equation}
When $\eps$ goes to 0, we get that $\Psiin(\alpha)=\Psiout(\alpha)$. 

\bibliographystyle{amsplain}

\end{document}